\documentclass[12pt]{amsart}

\usepackage{amsmath, amsfonts,amssymb,  amscd, amsthm, amsxtra,verbatim}

\usepackage{url}
\usepackage[all]{xy}

\DeclareMathOperator{\Spec}{Spec}

\DeclareMathOperator{\sing}{sing}

\DeclareMathOperator{\Coh}{Coh}

\def\A{\mathbb{A}}

\def\P{\mathbb{P}}

\def\m{\mathfrak{m}}

\def\onto{\twoheadrightarrow}
\def\into{\hookrightarrow}
\def\xto{\xrightarrow}

\def\H{\mathcal{H}}

\def\I{\mathcal{I}}

\def\R{\mathcal{R}}

\theoremstyle{plain}
\newtheorem{theorem}{Theorem}[section]
\newtheorem{lemma}[theorem]{Lemma}
\newtheorem{corollary}[theorem]{Corollary}
\newtheorem{proposition}[theorem]{Proposition}

\theoremstyle{definition}

\theoremstyle{remark}
\newtheorem{remark}[theorem]{Remark}
\newtheorem{example}[theorem]{Example}

\begin{document}

\title{The Hilbert polynomial of a symbolic square}


\author{Kaloyan Slavov}

\begin{abstract}
Let $k$ be an algebraically closed field, and let $C\subset \P^n_k$ be a reduced closed subscheme with ideal sheaf $\I$. Let  $\I^{<2>}$ be the second symbolic power of $\I$. 
When $C$ is an integral curve, we compute the Hilbert polynomial of $\mathcal{O}_{\P^n}/\I^{<2>}$ in terms of invariants of $C$. 
\end{abstract}

.


\maketitle

\section{Introduction}
\label{sec_introd}

Fix an algebraically closed field $k$.  
Let $i\colon C\into\P^n_k$ be a reduced closed subscheme, and let $\I$ be its
ideal sheaf. Let $S=k[x_0,\dots,x_n]$ with the usual grading. 

Following Lazarsfeld (see p.~177 in \cite{Laz1} or 
p.~164 in \cite{Laz2}), define
the second symbolic power $\I^{<2>}$ of $\I$ as the ideal sheaf consisting of germs of functions that vanish to order at least $2$ at every point of $C$.
When $C$ is integral, Corollary 1 in \cite{EH} implies that 
for any affine open $U=\Spec R\subset\P^n$, the ideal $\Gamma(U,\I^{<2>})\subset R$ coincides with the classical second symbolic power of 
$\Gamma(U,\I),$ as traditionally defined in commutative algebra 
--- i.e., it is the primary component of the ordinary second power of $\Gamma(U,\I)$.
See also the discussion in Section 3.9 in \cite{E}. 
When $C$ is a local complete intersection, we have 
$\I^{<2>}=\I^2$; however, this can fail for other subschemes $C$ (see Example \ref{non_lci} in Section \ref{undcondsubsection}).

The study of the symbolic powers of various closed subschemes has received recent interest; see, e.g., \cite{CM}, \cite{Linear_symb}, \cite{star}, \cite{generator_symb}. See \cite{KlU} for another instance when the second symbolic power of the ideal of a subscheme is of interest, in connection with the birational projection of a variety onto a hypersurface.

The motivation to investigate the Hilbert polynomial of $\mathcal{O}_{\P^n}/\I^{<2>}$ when $\dim C\geq 1$ came up as we were working on the problem of
identifying the dimension of the moduli space of hypersurfaces $V(F)$ of degree $l$ in $\P^n$ whose singular locus has dimension at least $b$, for a given $b\geq 1$; see \cite{S}. Namely, for a fixed reduced $C\into\P^n$ with ideal sheaf $\I,$ Lazarsfeld's definition of $\I^{<2>}$ directly implies that for $F\in S_l,$ we have $C\subset V(F)_{\sing}$ if and only if $F\in\Gamma(\P^n,\I^{<2>}(l)).$
In \cite{S}, we give an upper bound for 
$\dim \Gamma(\P^n,\I^{<2>}(l))$ that is valid for all $l$.  

Our first task now is to describe the sheaf $\I^{<2>}$ in a more explicit manner. 
Namely, in Section \ref{undcondsubsection}, we recall that
$\I^{<2>}$ fits into the exact sequence
\[
0\to\I^{<2>}\to\I\to\Omega_{\P^n}/\I\Omega_{\P^n}\to i_*\Omega_C\to 0.
\]

In Section \ref{hilbpolyJsubsect}, we turn to the natural question of computing the Hilbert polynomial of 
$\mathcal{O}_{\P^n}/\I^{<2>}$. 
Let $P_C(l)$ be the Hilbert polynomial of $C$.
We prove the following

\begin{proposition} The Hilbert polynomial of $\mathcal{O}_{\P^n}/\I^{<2>}$ is given by
\[\dim_k\Gamma(\P^n,(\mathcal{O}_{\P^n}/\I^{<2>})(l))=(n+1)P_C(l-1)-\dim_k\Gamma(C,\Omega_C(l))\]
for $l\gg 0.$
\label{Hilb_poly_O_mod_J}
\end{proposition}

The goal of Section \ref{section_curve} is to refine Proposition \ref{Hilb_poly_O_mod_J} in the case when $C$ is an integral curve. 
For this, we need to investigate the term $\Gamma(C,\Omega_C(l))$ in more detail. 

Let $C$ be an integral curve over $k$, and let $p\colon \widetilde{C}\to C$ be its normalization.
Let $\Psi\colon \Omega_C\to p_*\Omega_{\widetilde{C}}$ be the canonical map, and let
$\R_1,\R_2$ denote its kernel and cokernel:
\[0\to\R_1\to\Omega_C\xto{\Psi}p_*\Omega_{\widetilde{C}}\to\R_2\to 0.\]
Since $p$ is an isomorphism over a dense open subset of $C,$ so is $\Psi,$ and hence $\R_1$ and $\R_2$ have finite support,
contained in $C_{\sing}$. 
In particular, $\R_1$ and $\R_2$ are torsion sheaves. Note that the torsion subsheaf of $\Omega_C$ is contained in $\R_1,$ since 
$p_*\Omega_{\widetilde{C}}$ is locally free, as a $p_*\mathcal{O}_{\widetilde{C}}$-module; therefore, 
$\R_1$ is the torsion subsheaf of $\Omega_C$. See the proof of Proposition 2.2 in \cite{EKl} for its rich history. 
 For each $P\in C_{\sing},$ the stalks $(\R_1)_P$ and $(\R_2)_P$ are finite-dimensional $k$-vector spaces, so we can define the
 invariant
\[\mu(C):=\sum_{P\in C_{\sing}}(\dim_k (\R_1)_P-\dim_k(\R_2)_P).\]

For an integral curve $i\colon C\into\P^n$ with ideal sheaf $\I$ and saturated ideal $I,$
we let $d$ be its degree and $p_a$ be its arithmetic genus, so for large $l,$ we have
\begin{equation}
P_C(l)=\dim_k(S/I)_l=dl+1-p_a.
\label{Hilb_poly_curve}
\end{equation}

Notation as in \cite{EKl}, let $p_g$ denote the geometric genus of $C$, i.e., the genus of
$\widetilde{C}$. We prove

\begin{proposition}
For $l\gg 0,$
\[\dim_k\Gamma(\P^n, (\mathcal{O}_{\P^n}/\I^{<2>})(l)     )=ndl+1+(n+1)(1-d-p_a)-p_g-\mu(C).\]
\label{compHpolyprop}
\end{proposition}

When $C$ is an integral curve that is a local complete intersection, Proposition 2.2 in \cite{EKl} implies that 
$\chi(\Omega_C)=\chi(\omega_C),$ where $\omega_C$ is the dualizing sheaf of $C$. As a consequence,
\[\mu(C)=\chi(\Omega_C)-\chi(\Omega_{\widetilde{C}})=\chi(\omega_C)-\chi(\Omega_{\widetilde{C}})=p_a-p_g.\]
We recover this result in the particular case when $C$ is a plane curve, by an elementary argument. Namely,
we compute the codimension of $\Gamma(\P^n,\I^{<2>}(l))$ in $S_l$ directly and compare with the formula in Proposition
\ref{compHpolyprop}. As an easy corollary, we deduce:

\begin{corollary}
For an integral plane curve $C\into\P^2\subset\P^n,$ the Hilbert
polynomial of the sheaf $\Omega_C$ of K\"ahler differentials is
\[\chi(\Omega_C(l))=dl+p_a-1.\]
\label{kahler_plane_curves}
\end{corollary} 

Note that $C$ is not required to be smooth. 

Finally, in Section \ref{eg_section}, we compute explicitly the invariant $\mu(C)$ in the following 

\begin{example}
Consider the ideal 
\[I=(y^2-xz,yz-x^3,z^2-x^2y)\subset A=k[x,y,z],\] 
and let $C^0=V(I)\subset\A^3$; this is 
the curve parametrized by \[x=t^3, y=t^4, z=t^5.\] The 
projective closure $C\subset\P^3$ of $C^0$ has 
$d=5,$ $p_a=2$, and $p_g=0.$ 
The unique singular point of $C$ is $(0,0,0)\in C^0$.
\label{eg_intro}
\end{example}

In Section \ref{eg_section}, 
we compute that $\mu(C)=3,$ so the Hilbert polynomial of $\mathcal{O}_{\P^3}/\I^{<2>}$ is $15l-26$.

\section{An exact sequence for $\I^{<2>}$}
\label{undcondsubsection}

We investigate the sheaf $\I^{<2>}$. Proposition \ref{exact_seq_symb_power} below is a particular case of Exercise 8 on p.~83 in \cite{Kunz}. We include its proof for completeness.  

\begin{proposition} 
\begin{itemize}
\item[a)] Let $I\subset A=k[x_1,\dots,x_n]$ be a radical ideal. Then 
\[I^{<2>}=\text{Ker}\left(I\to\Omega_A/I\Omega_A\right).\]
\item[b)] Let $i:C\into\P^n$ be a reduced closed subscheme with ideal sheaf $\I.$ There is an exact sequence
\begin{equation}
0\to\I^{<2>}\to\I\to\Omega_{\P^n}/\I\Omega_{\P^n}\to i_*\Omega_C\to 0.
\label{eqn_exact_seq_for_symb_power}
\end{equation}
\end{itemize}
\label{exact_seq_symb_power}
\end{proposition}

\begin{proof}
It suffices to prove part a). 
Suppose that $f\in I$ satisfies $f\in\m^2$ for all $\m\supset I.$
 We claim that $df\in I\Omega_{A/k},$
where $d\colon A\to\Omega_{A/k}$ is the canonical derivation. (In this way, we linearize the a-priori inconvenient condition that $f\in\m^2$).
 We know that for each maximal $\m\supset I$, we have
$\Omega_{A/k}/\m\Omega_{A/k}=\Omega_{A/k}\otimes_A A/\m\simeq\m/\m^2$ as $A/\m$-vector spaces, and
\[\xymatrix{
I/I^2\ar[r]\ar[d]_{d} & \m/\m^2\ar[d]_{\simeq}\\
\Omega_{A/k}/ I\Omega_{A/k}\ar[r] & \Omega_{A/k}/\m\Omega_{A/k}
}
\]
commutes, so the condition $f\in\m^2$ is equivalent to $df\in \m\Omega_{A/k}.$

Since $\Omega_{A/k}$ is a {\it free} $A$-module, we conclude that
\[df\in \bigcap_{\m\supset I} (\m\Omega_{A/k})=\left(\bigcap_{\m\supset I} \m\right)\Omega_{A/k}=I\Omega_{A/k}.\]

The converse is obvious from the commutative diagram above.
\end{proof}

\begin{remark} Notation as in Example \ref{eg_intro},
consider
\[f=xy(y^2-xz)-x^2(yz-x^3)+z(z^2-x^2y)\in I.\] An easy check shows that $df=0$ in $\Omega_A/I\Omega_A$, so $f\in I^{<2>}.$ However, every element of $I^2$ contains only monomials of degree $4$ or larger, and since $f$ contains the monomial $z^3$ of degree $3$, we deduce that
$f\notin I^2.$ Thus, $I^2$ is strictly smaller than 
$I^{<2>}$ in this example. 
\label{non_lci} 
\end{remark}

\begin{corollary}
Suppose that $C\subset\P^n$ is a reduced closed subscheme, which is a local complete intersection. Let $\I$ be the ideal sheaf of $C$. 
Then for $F\in S_l$ we have $C\subset V(F)_{\sing}$ if and only if $F\in \Gamma(\P^n,\I^2(l)).$
\end{corollary}

\begin{proof}
 The condition $C\subset V(F)_{\sing}$ is equivalent to
$F\in\Gamma(\I^{<2>}(l)).$
The inclusion $\I^2\subset\I^{<2>}$ is an equality if $C$ is a local complete intersection,
since in this case, the map $i^*\I\to i^*\Omega_{\P^n}$ is injective (e.g.~Exercise 16.17 in \cite{E}).
\end{proof}

\section{The Hilbert polynomial of $\mathcal{O}_{\P^n}/\I^{<2>}$}
\label{hilbpolyJsubsect}

Let $i:C\into\P^n$ be a reduced closed subscheme with ideal sheaf $\I$.

\begin{proof}[Proof of Proposition \ref{Hilb_poly_O_mod_J}.]
Define $\H\in \Coh(C)$ by the exactness of
\begin{equation}
0\to\H\to i^*\Omega_{\P^n}\to\Omega_{C}\to 0,
\label{defofH}
\end{equation}
so (\ref{eqn_exact_seq_for_symb_power}) yields a short exact sequence
\begin{equation}
0\to\I^{<2>}\to \I\to i_*\H\to 0.
\notag
\end{equation}

 From the short exact  sequence
\[0\to i_*\H\to \mathcal{O}_{\P^n}/\I^{<2>}\to\mathcal{O}_{\P^n}/\I\to 0,\]
we obtain a short exact sequence
\[0\to{\Gamma}(\P^n,i_*\H(l))\to{\Gamma}(\P^n,(\mathcal{O}_{\P^n}/\I^{<2>})(l))\to
{\Gamma}(\P^n,(\mathcal{O}_{\P^n}/\I)(l))\to 0\]
for large $l$. The last term has dimension
$P_C(l)$ for large $l,$ so it
suffices to compute the dimension of the first term.

The short exact sequence
\[0\to i_*\H\to \Omega_{\P^n}/\I\Omega_{\P^n}\to i_*\Omega_C\to 0\]
gives rise to a short exact sequence
\[0\to {\Gamma}(\P^n,i_*\H(l))\to{\Gamma}(\P^n,(\Omega_{\P^n}/\I\Omega_{\P^n})(l))\to\Gamma(C,\Omega_{C}(l))\to 0\]
for $l$ sufficiently large. 

Finally, we have to compute 
$\text{dim}_k{\Gamma}(\P^n,(\Omega_{\P^n}/\I\Omega_{\P^n})(l))$ for large $l$. Recall
(e.g.~Theorem II.8.13 in \cite{H}) the
short exact sequence
\begin{equation}
0\to\Omega_{\P^n}\to \mathcal{O}_{\P^n}(-1)^{\oplus (n+1)}\to \mathcal{O}_{\P^n}\to 0.
\label{Euler}
\end{equation}
Since $\mathcal{O}_{\P^n}$ is locally free, applying $-\otimes_{\mathcal{O}_{\P^n}}\mathcal{O}_{\P^n}/\I$ to this short exact sequence yields
a short exact sequence
\[0\to\frac{\Omega_{\P^n}}{\I\Omega_{\P^n}}\to \left(\frac{\mathcal{O}_{\P^n}(-1)}{\I\mathcal{O}_{\P^n}(-1)}\right)^{\oplus (n+1)}\to
\frac{\mathcal{O}_{\P^n}}{\I}\to 0.\]
The sequence
\begin{equation}
0\to{\Gamma}\left(\frac{\Omega_{\P^n}}{\I\Omega_{\P^n}}(l)\right)\to
{\Gamma}\left(\frac{\mathcal{O}_{\P^n}(-1)}{\I\mathcal{O}_{\P^n}(-1)}(l)\right)^{\oplus (n+1)}\to
{\Gamma}\left(\frac{\mathcal{O}_{\P^n}}{\I}(l)\right)\to 0
\label{gamma_tilde_om_mod_i}
\end{equation}
is exact for large $l$. For large $l$, the third term has dimension $P_C(l)$ as before.

We are left to compute $\text{dim}_k{\Gamma}(\P^n,(\mathcal{O}(-1)/\I\mathcal{O}(-1))(l))$ for large $l$.
Notice that $\I\mathcal{O}(-1)\simeq\I(-1)$ and that for large $l,$ \[{\Gamma}(\P^n,(\mathcal{O}(-1)/\I(-1))(l))=
{\Gamma}\left(\P^n,(\mathcal{O}_{\P^n}/\I)(l-1)\right)\] has dimension $P_C(l-1)$.

Going back through the exact sequences, we complete the calculation.
\end{proof}

For example, if $C\simeq\P^r$ is an $r$-dimensional projective linear subspace, we know the Hilbert polynomial $P_C(l)$, and we can easily determine
the dimensions $\dim_k\Gamma(\P^r,\Omega_{\P^r}(l))$ for large $l$, by using the Euler sequence (\ref{Euler}) for $\Omega_{\P^r}.$ This computes the Hilbert polynomial of $\mathcal{O}_{\P^n}/\I^{<2>}$, and hence the Hilbert polynomial of $\I^{<2>}$. This approach
generalizes to the case when $C$ is a disjoint
union of linear subspaces, hence we obtain a weak version of Lemma 2.3 in \cite{Linear_symb}.

As an easy special case, let $C=\{P_1,\dots,P_d\}$ be a finite set of points (with reduced induced structure). 
 Since $\Omega_C=0$, we have  
$\H\simeq i^*\Omega_{\P^n},$ and hence an exact sequence
\[0\to\Omega_{\P^n}/\I\Omega_{\P^n}\to \mathcal{O}_{\P^n}/\I^{<2>}\to\mathcal{O}_{\P^n}/\I\to 0.\]
However, $\Omega_{\P^n}/\I\Omega_{\P^n}=\Omega_{\P^n}\otimes_{\mathcal{O}_{\P^n}}(\mathcal{O}_{\P^n}/\I)$ has zero-dimensional support (the same is true for all of its
twists by $\mathcal{O}_{\P^n}(l)$), hence 
\[H^1((\Omega_{\P^n}/\I\Omega_{\P^n})(l))=0.\] Thus, for all $l,$ we have an exact sequence
\[0\to{\Gamma}\left(
(\Omega_{\P^n}/\I\Omega_{\P^n})(l)\right)\to
{\Gamma}\left((\mathcal{O}_{\P^n}/\I^{<2>})(l)\right)\to
{\Gamma}\left((\mathcal{O}_{\P^n}/\I)(l)\right)\to 0.\] This vanishing of $H^1(\P^n,(\Omega_{\P^n}/\I\Omega_{\P^n})(l))$ also
implies that the sequence (\ref{gamma_tilde_om_mod_i}) is exact for all $l$.
Therefore,
\[\dim_k\Gamma((\mathcal{O}_{\P^n}/\I^{<2>})(l))=(n+1)\dim_k\Gamma((\mathcal{O}_{\P^n}/\I)(l-1))=(n+1)d\]
for all $l$. See \cite{generator_symb} for a discussion of the more subtle question of the Hilbert function of the saturated ideal of
$\I^2$.

\section{The case when $C$ is a curve}
\label{section_curve}

\begin{lemma} Let 
$i\colon C\into\P^n$ be an integral curve. Notation as
in Section \ref{sec_introd}, we have
\[\dim_k\Gamma(C,\Omega_{C}(l))=dl+p_g-1+\mu(C)\quad\text{for}\ l\gg 0.\]
\label{EcharOmega}
\end{lemma}

\begin{proof}
For large $l$, the sequence
\begin{equation}
0\to\Gamma(C,\R_1(l))\to\Gamma(C,\Omega_C(l))\to\Gamma(C,(p_*\Omega_{\widetilde{C}})(l))\to\Gamma(C,\R_2(l))\to 0
\label{HpolyOme}
\end{equation}
is exact.

Note that
\[\Gamma(C,\R_1(l))\simeq\Gamma(C,\R_1)=\bigoplus_{P\in C_{\sing}} (\R_1)_P,\]
and similarly for $\R_2.$

Now, we look at the term $\Gamma(C, (p_*\Omega_{\widetilde{C}})(l)).$ By the projection formula, we know
\[(p_*\Omega_{\widetilde{C}})(l)\simeq p_*(\Omega_{\widetilde{C}}\otimes_{\mathcal{O}_{\widetilde{C}}}p^*\mathcal{O}_C(l)).\]
Since $C$ has degree $d$,
$p^*\mathcal{O}_C(l)$ is a line bundle on $\widetilde{C}$ of degree $dl$
(e.g.~Proposition 3.8 on p.~276 in \cite{L}).
By the Riemann-Roch theorem applied to $\widetilde{C},$ it follows that for large $l$,
\[\dim_k\Gamma(\widetilde{C},\Omega_{\widetilde{C}}\otimes p^*\mathcal{O}C(l))=dl+p_g-1.\]
Taking the alternating sum of dimensions in (\ref{HpolyOme}) completes the proof.
\end{proof}

\begin{proof}[Proof of Proposition \ref{compHpolyprop}.]
Combine Proposition \ref{Hilb_poly_O_mod_J}, Lemma \ref{EcharOmega}, and (\ref{Hilb_poly_curve}).
\end{proof}

Let $I=(f,x_{b+2},\dots,x_n)\subset S,$ where $f\in k[x_0,\dots,x_{b+1}]_d-\{0\}$ (for us, the important case will be $b=1$).
Consider the (surjective) composition
\[\Phi\colon  k[x_0,\dots,x_{b+1}]_l\oplus\left(\bigoplus_{i=b+2}^n k[x_0,\dots,x_{b+1}]_{l-1}x_i\right)
\into S_l
\onto S_l/(I^2\cap S_l).\]

\begin{lemma}
 We have that
\[\ker(\Phi)=\{P+\sum_{i=b+2}^n P_i x_i\ \colon\ f^2|P, f|P_i\ \text{for}\ i=b+2,\dots,n\}.\]
For $l\geq 2d$, the codimension of $I^2_l=I^2\cap S_l$ in $S_l$ equals $\beta_d(l)$, where
\begin{multline}
\beta_d(l)
=\binom{l+b+1}{b+1}-\binom{l-2d+b+1}{b+1}+\\
(n-b-1)\left(\binom{l+b}{b+1}-\binom{l-d+b}{b+1}\right).
\end{multline}
\label{explct}
\end{lemma}

\begin{proof} If $P+\sum P_i x_i\in I^2,$ just expand it as a polynomial in $x_{b+2},\dots,x_n$.
The second part is an immediate consequence.
\end{proof}

\begin{corollary}
For an integral {\it plane} curve $C$, we have $\mu(C)=p_a-p_g$.
\label{cor_mu_plane_curve}
\end{corollary}

\begin{proof}
We compute the codimension of $\Gamma(\P^n,\I^{<2>}(l))$ in $S_l$ for large $l$ in two different ways. On the one hand, it is given by the formula in Proposition
\ref{compHpolyprop}. On the other hand, since $C$ is a (local) complete intersection, we know that $\I^{<2>}=\I^2,$ and since the ideal $I^2$ is saturated,
we have, explicitly, $\Gamma(\P^n,\I^{<2>}(l))=I^2_l$. Thus, the codimension we are computing equals the codimension of $I^2_l$ in $S_l,$ which we computed as $\beta_d(l)$ in  Lemma \ref{explct} (take $b=1$). We equate the two linear polynomials in $l$ and compare their constant coefficients. Recall that $p_a=\frac{(d-1)(d-2)}{2}$ to obtain the desired conclusion.
\end{proof}

\begin{remark}
The conclusion of Corollary \ref{cor_mu_plane_curve} can fail for a general integral curve $C\subset\P^n$. For instance, if $C$ is the curve defined in Example
\ref{eg_intro}, then $\mu(C)=3$ by Lemma \ref{eg_mu} in Section \ref{eg_section}, while $p_a-p_g=2.$ 
\end{remark}

\begin{proof}[Proof of Corollary \ref{kahler_plane_curves}.]
Combine Lemma \ref{EcharOmega} and Corollary
\ref{cor_mu_plane_curve}.
\end{proof}

\section{An explicit example}
\label{eg_section}

Assume that $\text{char}(k)\neq 2,3,5.$
Notation as in Example \ref{eg_intro}, let 
$B=A/I$, and note that $\Omega_B=(Bdx\oplus Bdy\oplus Bdz)/(\eta_1,\eta_2,\eta_3),$ where
\begin{align*}
\eta_1 &=-zdx+2ydy-xdz,\\
 \eta_2 &=-3x^2dx+zdy+ydz, \\
 \eta_3&=2xydx+x^2dy-2zdz.
\end{align*}
Consider the map
 $\Psi:\Omega_B\to k[t]dt$ induced by $x\mapsto t^3, y\mapsto t^4, z\mapsto t^5$. 
The cokernel of $\Psi$ has dimension $2$ over $k$;
now we investigate $\ker(\Psi)$:

\begin{lemma} 
We have $\dim_k \ker(\Psi)=5.$
\label{eg_mu}
\end{lemma} 

\begin{proof}
First note that any element $b\in B$ can be written uniquely as 
\begin{equation}
b=\alpha(x)+\beta(x)y+\gamma(x)z,
\label{decomp_elt_b}
\end{equation}
where
$\alpha,\beta,\gamma\in k[x].$ Indeed, existence follows by using the relations in $B$, and for uniqueness, suppose that an element $\alpha(x)+\beta(x)y+\gamma(x)z$ of $A$ (with $\alpha,\beta,\gamma\in k[x]$) belongs to $I$. Since $I$ contains only monomials of degree $2$ or larger, it follows that $\alpha=x\alpha_1, \beta=x\beta_1, \gamma=x\gamma_1$. But then, since $I$ is prime, 
 $\alpha_1+\beta_1 y+\gamma_1 z$ must be in $I$. Continuing the process, we deduce that $\alpha,\beta,\gamma\in (x^n)$ for all $n$, hence $\alpha=\beta=\gamma=0.$ 
 
Consider a differential \[w=(P_1+Q_1 y+R_1 z)dx+(P_2+Q_2 y + R_2 z)dy+(P_3+Q_3 y+R_3 z)dz\] in $\Omega_B$,
where all $P_i,Q_i,R_i$ are in $k[x]$.
Split the image of $w$ under the composition 
$\Omega_B\xto{\Psi}k[t]dt\xto{\simeq}k[t]$ as a sum of three polynomials according to the residues modulo $3$ of the exponents of the monomials that they contain. 
We deduce that $w$ belongs to $\ker(\Psi)$ if and only if
\begin{gather*}
3P_1(x)+4x^2R_2(x)+5x^2Q_3(x)=0\\
3xQ_1(x)+4P_2(x)+5x^2R_3(x)=0\\
3xR_1(x)+4xQ_2(x)+5P_3(x)=0,\\
\end{gather*}
in $k[x]$, i.e.,
\begin{gather*}
P_1(x)=-\frac{1}{3}\left(4x^2R_2(x)+
5x^2Q_3(x)\right)\\
P_2(x)=-\frac{1}{4}\left(3xQ_1(x)+5x^2R_3(x)\right)\\
P_3(x)=-\frac{1}{5}\left(3xR_1(x)+4xQ_2(x)\right).\\
\end{gather*}
Going back to the expression of $w$, it follows that
$\Psi(w)=0$ if and only if $w$ is a $k[x]$-linear combination of the following $6$ differentials in $\Omega_B$:
\begin{gather*}
w_1:=5zdx-3xdz,
w_2:=-5x^2dx+3ydz,
w_3:=4ydx-3xdy,\\
w_4:= -4x^2dx+3zdy,
w_5:=-5x^2dy+4zdz, w_6:=5ydy-4xdz.
\end{gather*}
However, since $w_4=-w_2,$
$w_5=xw_3$ and $w_6=\frac{w_1}{2}$, it follows that $\ker(\Psi)$ is the $k[x]$-submodule of $\Omega_B$ generated by $w_1,w_2,w_3.$ Next, it is easy to see that $x^2w_1=0, xw_2=0, x^2w_3=0$ in $\Omega_B$. Therefore, the $5$ differentials $w_1, xw_1, w_2, w_3, xw_3$ span $\ker(\Psi)$ as a $k$-vector space. 
 
Suppose that $a_1,\dots,a_5\in k$ give a linear dependence relation among these $5$ differentials.  Working in $Bdx\oplus Bdy\oplus Bdz,$ this means that there exist $\alpha_i,\beta_i,\gamma_i\in k[x]$, for $i=1,2,3,$ such that
\[a_1w_1+a_2xw_1+a_3w_2+a_4w_3+a_5xw_3=\sum_{i=1}^3
(\alpha_i+\beta_i y+\gamma_i z)\eta_i\] 
(by abuse of notation, the obvious lift of $w_i$ is still denoted by $w_i$). 
Comparing the coefficients of $dy,dz,$ and then in turn using uniqueness of the decomposition (\ref{decomp_elt_b}), we obtain the following equalities in $k[x]$:
\begin{align}
-3(a_4+a_5 x)x &=2x^3\gamma_1+x^3\beta_2+x^2\alpha_3    \label{xdy}\\
0 &=2\alpha_1+x^2\gamma_2+x^2\beta_3 \label{ydy}\\
0 &=\alpha_2+2x\beta_1+x^2\gamma_3    \label{zdy}\\
-3(a_1+a_2 x)x &=-x\alpha_1 +x^3\gamma_2-2x^3\beta_3    \label{xdz}\\
3a_3 &=-x\beta_1+\alpha_2-2x^2\gamma_3    
\label{ydz}\\
0 &=-2\alpha_3-x\gamma_1+x\beta_2 \label{zdz}.
\end{align}

By (\ref{ydy}), we have $\alpha_1\in (x^2),$ hence
(\ref{xdz}) implies that $a_1=a_2=0.$ Next, by
(\ref{zdy}), we have $\alpha_2\in (x),$ so (\ref{ydz}) yields $a_3=0.$ Finally, (\ref{zdz}) implies that $\alpha_3\in (x),$ hence (\ref{xdy}) gives $a_4=a_5=0.$
\end{proof}

\section{Further study}

We can list a number of questions related to the current study. For example, given a closed subscheme
$C\into\P^n$ (not necessarily reduced), it would be interesting to describe the space of 
all hypersurfaces $F\in S_l$ for which there is a scheme-theoretic inclusion
\[C\into V\left(F,\frac{\partial F}{\partial x_0},\dots,\frac{\partial F}{\partial x_n}\right):=\text{Proj}\left(S/(F,\frac{\partial F}{\partial x_i})\right).\]

Also, given an integral curve $C\subset\P^n,$ it would be interesting to investigate the Hilbert polynomials of higher symbolic powers of $\I$.
For example, if we mimic the discussion in Section \ref{undcondsubsection}, we find that the third symbolic power $\I^{<3>}$ fits into an exact sequence
\[0\to\I^{<3>}\to\I^{<2>}\to\Omega_{\P^n}/\I^{<2>}\Omega_{\P^n}\to j_*\Omega_D\to 0,\]
where $j:D\into\P^n$ is the closed subscheme whose ideal sheaf is $\I^{<2>}$ (of course, $C\simeq D_{\text{red}}$). 
However, the Hilbert polynomial $\chi(\Omega_D(l))$ is more difficult to analyze. 

Finally, when $C\subset\P^n$ is a curve, the question of the Hilbert {\it function} of $I^{<2>}$ (the saturated ideal of $\I^2$)
 appears naturally. As suggested, for example,
in  \cite{CM} and \cite{generator_symb}, this question will be interesting and non-trivial, since it is such already when $C$ is a finite set of points.

\section*{Acknowledgments}

The problem that we address in this article came up naturally as I was working on my doctoral thesis at MIT under the direction of Bjorn Poonen. I am grateful
to Prof. Poonen for all of his dedication and substantial help throughout the entire process. I also thank an anonymous referee 
for some helpful comments and suggestions. I thank Steven Kleiman for a number of references and thorough comments. Finally, I thank Martin Kreuzer for some additional references.

\end{document}